\RequirePackage{fix-cm}
\documentclass[11pt,reqno]{amsart}
\usepackage{amsmath,amssymb,amsthm}
\usepackage{amsfonts}
\usepackage{mathtools}
\usepackage{enumitem}

\usepackage{xcolor}
\usepackage[pdftex,breaklinks]{hyperref}
\hypersetup{colorlinks=true,linkcolor=blue,citecolor=red}

\usepackage[normalem]{ulem}

\numberwithin{equation}{section}

\theoremstyle{plain}
\newtheorem{theorem}{Theorem}[section]
\newtheorem{lemma}[theorem]{Lemma}

\newtheorem{proposition}[theorem]{Proposition}

\newtheorem*{theorem-non}{Theorem}

\theoremstyle{definition}

\newtheorem{remark}[theorem]{Remark}

\theoremstyle{definition}
\newtheorem{hypothesis}[theorem]{Hypothesis}

\theoremstyle{definition}

\newcommand{\be}{\begin{equation}}
\newcommand{\ee}{\end{equation}}
\newcommand{\n}{\noindent}

\DeclareMathOperator{\loc}{loc}

\DeclareMathOperator{\BMO}{BMO}

\DeclareMathOperator{\tr}{trace}
\DeclareMathOperator{\AM}{AM}
\DeclareMathOperator{\Var}{Var}

\DeclareMathOperator{\Lin}{Lin}

\newcommand{\tsn}[1]{{[\kern-0.4ex] #1 [\kern-0.4ex]}}

\newcommand{\Ksubset}{{\subset\!\subset}}

\newcommand{\grad}{\nabla}

\newcommand{\bB}{{\mathbf B}}

\newcommand{\bH}{{\mathbf H}}

\newcommand{\bF}{{\mathbf F}}
\newcommand{\bG}{{\mathbf G}}
\newcommand{\bK}{{\mathbf K}}
\newcommand{\bu}{{\mathbf u}}

\newcommand{\bv}{{\mathbf v}}
\newcommand{\bx}{{\mathbf x}}
\newcommand{\bd}{{\mathbf d}}

\newcommand{\bz}{{\mathbf z}}

\newcommand{\bw}{{\mathbf w}}

\newcommand{\bs}{{\mathbf s}}

\newcommand{\bb}{{\mathbf b}}
\newcommand{\bA}{{\mathbf A}}

\newcommand{\sS}{{\mathcal S}}
\newcommand{\sD}{{\mathcal D}}

\newcommand{\sK}{{\mathcal K}}

\newcommand{\R}{{\mathbb R}}

\newcommand{\Z}{{\mathbb Z}}

\newcommand{\MN}{{\mathbb R}^{N\times n}}

\newcommand{\bue}{{{\mathbf u}\e}}

\newcommand{\Om}{{\Omega}}

\newcommand{\E}{{\mathcal E}}

\newcommand{\e}{_{\mathrm e}}

\newcommand{\dd}{{\mathrm d}}
\newcommand{\DD}{{{\mathrm D}}}

\newcommand{\rmT}{{\mathrm T}}

\def\Xint#1{\mathchoice
   {\XXint\displaystyle\textstyle{#1}}
   {\XXint\textstyle\scriptstyle{#1}}
   {\XXint\scriptstyle\scriptscriptstyle{#1}}
   {\XXint\scriptscriptstyle\scriptscriptstyle{#1}}
   \!\int}
\def\XXint#1#2#3{{\setbox0=\hbox{$#1{#2#3}{\int}$}
     \vcenter{\hbox{$#2#3$}}\kern-.5\wd0}}

\def\dashint{\Xint-}

\newlength{\extramargin}
\begin{document}

\title[Taylor's Theorem for Functionals on BMO]
{Taylor's Theorem for Functionals on BMO with Application
to BMO Local Minimizers}


\author[D. E. Spector]{Daniel E. Spector}
\address{Okinawa Institute of Science and Technology Graduate University,
Nonlinear Analysis Unit, 1919--1 Tancha, Onna-son, Kunigami-gun,
Okinawa, Japan}
\email{daniel.spector@oist.jp}
\thanks{}

\author[S. J. Spector]{Scott J. Spector}
\address{Department of Mathematics,
Southern Illinois University, Carbondale, IL 62901, USA}
\email{sspector@siu.edu}
\thanks{}


\date{24 May 2020}
\dedicatory{}

%
%



\begin{abstract}  In this note two results are established for
energy functionals
that are given by the integral   
of $ W(\bx,\grad\bu(\bx))$ over $\Om\subset\R^n$ with
$\grad\bu\in\BMO(\Om;\MN)$, the space of functions of
Bounded Mean Oscillation of John \& Nirenberg.
A version of Taylor's theorem is first shown to be valid provided the
integrand $W$ has polynomial growth.  This result is then used to
demonstrate  
that,
for the Dirichlet, Neumann, and mixed problems, every Lipschitz-continuous
solution of the corresponding Euler-Lagrange equations at which the
second variation of the energy is uniformly positive is a strict local
minimizer of the energy in $W^{1,\BMO}(\Om;\R^N)$,
the subspace of the Sobolev space $W^{1,1}(\Om;\R^N)$ for which the weak
derivative $\grad\bu\in\BMO(\Om;\MN)$.
\end{abstract}


\maketitle
\setlength{\parskip}{.5em}
\setlength{\parindent}{2em}
\baselineskip=15pt

\vspace{-1.25cm}

\section{Introduction}


Let $\Om\subset\R^n$, $n\ge 2$, be a Lipschitz domain.
Suppose that $\bd:\sD \to\R^N$, $N\ge1$, is a given
Lipschitz-continuous function, where  
$\sD\subset\partial\Om$,
the boundary of $\Om$. 
We herein consider functionals of the form
\be\label{eqn:energy-intro}
\E(\bu)= \int_\Om W\big(\bx,\grad\bu(\bx)\big)\, \dd\bx
\ee
\n for $W$ that satisfy, for some $a>0$ and $r>0$,
\[  
|\DD^3W(\bx,\bK)|\le a (1+|\bK|^r),
\]  
\n for all real $N$ by $n$  matrices $\bK$ and almost every
$\bx\in\Om$.  We take $\bu=\bd$ on $\sD$  and
$\bu\in W^{1,\BMO}(\Om;\R^N)$, the subspace of the Sobolev space
$W^{1,1}(\Om;\R^N)$ for which the weak derivative $\grad\bu$ is of
Bounded Mean Oscillation.  Our main result
shows that any Lipschitz-continuous weak solution $\bue$ of the
corresponding Euler-Lagrange equations:
\be\label{eqn:EL-intro}
0= \delta\E(\bue)[\bw]=
\int_\Om \DD
W\big(\bx,\grad\bue(\bx)\big)\big[\grad\bw(\bx)\big]\, \dd\bx\
\text{ for all } \bw\in\Var,
\ee
\n at which the second variation of $\E$ is uniformly positive:
for some $b>0$ and all $\bw\in\Var$,
\be\label{eqn:SV-intro}
 \delta^2\E(\bue)[\bw,\bw]=
\int_\Om \DD^2W\big(\bx,\grad\bue(\bx)\big)
\big[\grad\bw(\bx),\grad\bw(\bx)\big]\, \dd\bx
\ge
b\int_\Om |\grad\bw(\bx)|^2\,\dd\bx,
\ee
\n will satisfy, for some $c>0$,
\be\label{eqn:conclusion-intro}
\E(\bw+\bue)\ge \E(\bue) +c\int_\Om |\grad\bw(\bx)|^2\,\dd\bx
\ee
\n for all $\bw\in W^{1,\BMO}(\Om;\R^N)\cap\Var$ whose gradient
has sufficiently small \emph{norm} in $\BMO(\Om)$.  Here
\[  
\begin{gathered}
\DD^jW(\bx,\bK)= \frac{\partial^j}{\partial\bK^j}W(\bx,\bK),
\qquad
\Var:=\{\bw\in W^{1,2}(\Om;\R^N): \bw=\mathbf{0} \text{ on } \sD\},\\[4pt]
\|\grad\bu\|_{\BMO}:= \tsn{\grad\bu}_{\BMO}
+  \big| \langle\grad\bu\rangle_{\Om}  \big|,
\end{gathered}
\]  
\n  $\tsn{\cdot}_{\BMO}$ denotes the standard semi-norm on $\BMO(\Om)$
(see \eqref{eqn:BMO-V-norm}),
and $\langle\grad\bu\rangle_{\Om}$ denotes the average value of the
components of $\grad\bu$ on $\Om$.


The above result extends prior work\footnote{The
result in \cite[Section~6]{KT03} has been extended to the
Neumann and mixed problems in \cite[Section~3]{SS19}.}
 of Kristensen \&
Taheri~\cite[Section~6]{KT03} and
Campos Cordero~\cite[Section~4]{Ca17}
(see, also, Firoozye~\cite{Fi92}). 
These authors proved that, for the Dirichlet problem,
if $\bue$ is a Lipschitz-continuous weak solution of the
Euler-Lagrange equations, \eqref{eqn:EL-intro},
at which the second variation of $\E$ is uniformly positive,
\eqref{eqn:SV-intro}, then
there is a neighborhood of $\grad\bue$ in $\BMO(\Om)$ in which all
Lipschitz mappings have energy that is greater than 
the energy of $\bue$.

Our proof of the above result makes use of a version of
Taylor's theorem on $\BMO(\Om)$ that is established herein:
  Let $W$ satisfy,
for some $a>0$, $r>0$, and integer $k\ge2$,
\[  
|\DD^kW(\bx,\bK)|\le a (1+|\bK|^r),
\]  
\n for all real $N$ by $n$ matrices $\bK$, and almost every
$\bx\in\Om$.    Fix $M>0$ and
$\bF\in L^\infty(\Om;\MN)$.  Then there exists a constant
$c=c(M,||\bF||_\infty)>0$ such that
every $\bG\in\BMO(\Om;\MN)$ with
$||\bG-\bF||_{\BMO} <M$ satisfies
\be\label{eqn:Taylor-intro}
\int_\Om W(\bG)\,\dd\bx \ge
\int_\Om W(\bF)\,\dd\bx
+
\sum_{j=1}^{k-1} \frac1{j!}
\int_\Om \DD^j W(\bF)\big[\bH,\bH,\ldots,\bH\big]\,\dd\bx
-c \int_\Om |\bH|^k\,\dd\bx,
\ee
\n where $\bH=\bG-\bF$, $\bF=\bF(\bx)$, $\bG=\bG(\bx)$, and, e.g.,
$W(\bF)=W(\bx,\bF(\bx))$.


A key ingredient in our proof of \eqref{eqn:Taylor-intro} is the
\emph{interpolation inequality} \cite[Theorem~2.5]{SS19}:  If
$1\le p<q<\infty$, then there is a constant $C=C(p,q,\Om)$ such that,
for all $\psi\in\BMO(\Om)$,
\be\label{eqn:interp-intro}
\int_\Om |\psi(\bx)|^q\,\dd\bx
\le
C
\big(\tsn{\psi}_{\BMO}
+
\big| \langle\psi\rangle_{\Om}  \big|\big)^{q-p}
\int_\Om |\psi(\bx)|^p\,\dd\bx.
\ee
\n When $\Om=\R^n$ and $\langle\psi\rangle_{\R^n}=0$ this inequality is
due to Fefferman \& Stein~\cite[p.~156]{FS72}, although it is clear
from \cite[pp.~624--625]{Jo72}
that Fritz John was aware of \eqref{eqn:interp-intro} when
$\tsn{\psi}_{\BMO}$ was sufficiently small
and $\langle\psi\rangle_{\Om}=0$ (for domains $\Om$
with bounded eccentricity).


Finally, we note that our main result assumes that the solution $\bue$ of
the Euler-Lagrange equations \eqref{eqn:EL-intro} is Lipschitz continuous
and has uniformly positive second variation \eqref{eqn:SV-intro}.  It
follows that $\bue$ is a \emph{weak relative minimizer} of the energy
\eqref{eqn:energy-intro}, that is, a minimizer with respect to perturbations
that are small in $W^{1,\infty}$. Grabovsky \& Mengesha~\cite{GM07,GM09}
give further conditions\footnote{The most significant are quasiconvexity
in both the interior and at the boundary.
See Ball \& Marsden~\cite{BM84}.}  that they prove imply
that $\bue$ is then a \emph{strong relative minimizer} of $\E$,
that is, a minimizer with respect to to perturbations
that are small in $L^{\infty}$, whereas our result only changes
$W^{1,\infty}$ to $W^{1,\BMO}\subset\subset L^\infty$.   However,
as Grabovsky \& Mengesha have
noted, their results require that $\bue$ be $C^1$. Examples of
M\"{u}ller \& {\v{S}}ver{\'a}k~\cite{MS03} demonstrate that not all
Lipschitz-continuous solutions of \eqref{eqn:EL-intro} need be $C^1$.
Also, the Lipschitz-continuous example of
Kristensen \& Taheri~\cite[\S7]{KT03} satisfies both
\eqref{eqn:EL-intro} and \eqref{eqn:SV-intro}.

\section{Preliminaries}\label{sec:prelim}

For any \emph{domain} (nonempty, connected, open set)
$U\subset\R^n$, $n\ge2$, we denote by
$L^p(U;\R^N)$, $p\in[1,\infty)$, the space
of (Lebesgue) measurable functions $\bu$ with values
in $\R^N$, $N\ge1$,
whose $L^p$-norm is finite:
\[
||\bu||^p_{p}=||\bu ||^p_{p,U}
:= \int_U |\bu(\bx)|^p\,\dd\bx < \infty.
\]
\n $L^\infty(U;\R^N)$ shall denote those measurable
functions whose essential supremum is finite. We write
$L^1_{\loc}(U;\R^N)$ for the set of  measurable
functions that are integrable on every compact subset of $U$.


We shall write $\Om\subset\R^n$, $n\ge2$, to denote a
\emph{Lipschitz domain}, that is a
bounded domain whose boundary $\partial \Om$ is
(strongly) Lipschitz.
(See, e.g., \cite[p.~127]{EG92}, \cite[p.~72]{Mo66}, or
\cite[Definition~2.5]{HMT07}.)   Essentially, a bounded domain is
Lipschitz if, in a neighborhood of every $\bx\in\partial\Om$,  
the boundary is the graph of a
Lipschitz-continuous function and the domain is on ``one side''
of this graph.
$W^{1,p}(\Omega;\R^N)$ will denote the usual
Sobolev space of
functions $\bu\in L^p(\Omega;\R^N)$, $1\le p\le\infty$, whose
distributional gradient $\grad\bu$ is also contained in $L^p$.
Note that, since $\Om$ is a Lipschitz domain,
each $\bu\in W^{1,\infty}(\Omega;\R^N)$
has a representative that is Lipschitz continuous.
We shall write $\MN$ for the space of real $N$ by $n$ matrices with
  inner product $\bA:\bB=\tr(\bA\bB^\rmT)$ and norm
$|\bA|=\sqrt{\bA:\bA}$, where $\bB^\rmT$ denotes the transpose of $\bB$.


\subsection{Bounded Mean Oscillation}\label{sec:BMO}

The $\BMO$-seminorm\footnote{See
Brezis \& Nirenberg~\cite{BN95,BN96},
John \& Nirenberg~\cite{JN61}, Jones~\cite{Jo82},
Stein~\cite[\S4.1]{Stein-1993}, or, e.g.,
\cite[\S3.1]{Gr09}
for properties of $\BMO$.} of $\bF\in L^1_{\loc}(U;\MN)$
is given by
\be\label{eqn:BMO-V-norm}
\tsn{\bF}_{\BMO(U)}:=  \sup_{Q\Ksubset U}
 \dashint_Q |\bF(\bx)-\langle\bF\rangle_Q|\,\dd\bx,
\ee
\n where the supremum is to be taken over all
nonempty, bounded (open) $n$-dimensional
hypercubes
 $Q$  \emph{with faces parallel to the coordinate hyperplanes}.  Here
\[
\langle\bF\rangle_U
:=\dashint_U \bF(\bx)\,\dd\bx
:= \frac1{|U|}\int_U \bF(\bx)\,\dd\bx
\]
\n \emph{denotes the average value of the components of} $\bF$,
 $|U|$ denotes the
$n$-dimensional Lebesgue measure of any
bounded domain $U\subset\R^n$, and we write $Q\Ksubset U$ provided
that $Q\subset K_Q \subset U$ for some compact set $K_Q$.

The space $\BMO(U;\MN)$ (Bounded Mean Oscillation) is defined by
\be\label{eqn:BMO-V}
\BMO(U;\MN)
:=
\{\bF\in L_{\loc}^1(U;\MN): \tsn{\bF}_{\BMO(U)} < \infty\}.
\ee
\n One consequence of
\eqref{eqn:BMO-V-norm}--\eqref{eqn:BMO-V} is that
$L^\infty(U;\MN)\subset \BMO(U;\MN)$ with
\[  
\tsn{\bF}_{\BMO(U)} \le 2\|\bF\|_{\infty,U}\
\text{ for all $\bF\in L^\infty(U;\MN)$.}
\]  
\n We note for future reference that if $U=\Om$, a Lipschitz domain,
then a result of P.~W.~Jones~\cite{Jo82} implies, in particular, that
\[  
\BMO(\Om;\MN)\subset L^1(\Om;\MN).
\]  
\n It follows that\footnote{If $\bF=\grad\bw$ with $\bw=\mathbf{0}$ on
$\partial\Om$ then $\|\grad\bw\|_{\BMO}=\tsn{\grad\bw}_{\BMO(\Om)}$
since the integral of $\grad\bw$ over $\Om$ is then zero.}
\be\label{eqn:BMO-Om-norm}
\|\bF\|_{\BMO}:=\tsn{\bF}_{\BMO(\Om)} + | \langle\bF\rangle_{\Om}|
\ee
\n is a \emph{norm} on $\BMO(\Om;\MN)$.


\subsection{Further Properties of
\texorpdfstring{$\BMO$}{BMO}}\label{sec:maximal}

The main property of $\BMO$ that we shall use is contained
in the following result.  Although the proof can be found in
\cite{SS19}, the significant analysis it is based upon is due to
Fefferman \& Stein~\cite{FS72}, Iwaniec~\cite{Iw82}, and Diening,
R$\overset{_\circ}{\mathrm{u}}$\v zi\v cka, \& Schumacher~\cite{DRS10}.


\begin{proposition}\label{thm:main-2}  Let $\Om\subset\R^n$, $n\ge2$, be a
Lipschitz\footnote{This result, as stated, is valid for a larger
class of domains:  Uniform domains.  (Since $\BMO\subset L^1$
for such domains.  See P.~W.~Jones~\cite{Jo82},
Gehring \& Osgood~\cite{GO79}, and e.g., \cite{Ge87}.)  A slightly
modified version of this result is valid for John domains.
 See \cite{SS19} and the references therein.}
domain.  Then, for all $q\in[1,\infty)$,
\[
\BMO(\Om;\MN)\subset L^q(\Om;\MN)
\]
\n with continuous injection, i.e., there are constants
$J_1=J_1(q,\Om)>0$ such that,
for every $\bF\in\BMO(\Om;\MN)$,
\be\label{eqn:BMO-in-Lq}
\bigg(\,\dashint_\Om |\bF|^q\,\dd\bx\bigg)^{\!\!1/q}\!
\le
J_1\|\bF\|_{\BMO}.
\ee
\n Moreover, if  $1\le p<q<\infty$
then there exists
constants $J_2=J_2(p,q,\Om)>0$ such that every
$\bF\in\BMO(\Om;\MN)$ satisfies
\be\label{eqn:RH}
||\bF||_{q,\Om}
\le
J_2\Big(||\bF||_{\BMO}\Big)^{1-p/q}
\Big(||\bF||_{p,\Om}\Big)^{p/q}.
\ee
\n Here  $\|\cdot\|_{\BMO}$ is given by  (\ref{eqn:BMO-V-norm})
and  (\ref{eqn:BMO-Om-norm}).
\end{proposition}


\section{An Implication of Taylor's Theorem for a Functional on
\texorpdfstring{$\BMO$}{BMO}}

\begin{hypothesis}\label{def:W-TT} Fix $k,N\in\Z$ with $k\ge2$ and $N\ge1$.
We suppose that we are
given an \emph{integrand}
$W:\Omega\times\MN\to\R$
that satisfies:
\begin{enumerate}[topsep=-2pt]
\item[(H1)] $\bK\mapsto W(\bx,\bK) \in C^k(\MN)$, for $a.e.~\bx\in\Omega$;
\item[(H2)] $(\bx,\bK)\mapsto\DD^j W(\bx,\bK)$, $j=0,1,\ldots,k$, are each
(Lebesgue) measurable on their common domain
$\Omega\times\MN$; and
\item[(H3)] There are constants $c_k>0$ and $r>0$ such that,
for all $\bK\in\MN$ and $a.e.~\bx\in\Om$,
\[  
|\DD^k W(\bx,\bK)|\le c_k(1+|\bK|^r).
\]  
\end{enumerate}
\n Here, and in the sequel,
\[
\DD^0 W(\bx,\bK):=W(\bx,\bK), \qquad
\DD^j W(\bx,\bK)
:=
\frac{\partial^j}{\partial\bK^j}W(\bx,\bK)
\]
\n denotes $j$-th derivative of $\bK\mapsto W(\cdot,\bK)$. Note that,
for every $\bK\in\MN$, $a.e.~\bx\in\Om$, and $j=1,2,\ldots,k$,
\[ 
\DD^j W(\bx,\bK)\in
\Lin(\overbrace{\MN\times\MN\times\cdots\times\MN}^{\text{$j$ copies}} ;\R),
\] 
\n that is, $\DD^j W(\bx,\bK)$ can be viewed as a multilinear map from
$j$ copies of $\MN$ to $\R$.
\end{hypothesis}

\begin{remark}\label{rem:growth}  Hypothesis (H3) implies that each of
the functions, $\DD^j W$, $j=0,1,\ldots,k-1$, satisfies a similar growth
condition, i.e., 
$|\DD^{j} W(\bx,\bK)|\le c_j(1+|\bK|^{r+k-j})$.  It follows that each
of the functions $\DD^{j} W$ is (essentially) bounded on $\Om\times \sK$
for any compact $\sK\subset\MN$.
\end{remark}


\begin{lemma}\label{lem:Taylor}  Let $W$ satisfy Hypothesis~\ref{def:W-TT}.
Fix $M>0$ and
$\bF\in L^\infty(\Om;\MN)$.  Then there exists a constant
$c=c(M,||\bF||_\infty)>0$ such that
every $\bG\in\BMO(\Om;\MN)$ with
$||\bG-\bF||_{\BMO} <M$ satisfies
\be\label{eqn:Taylor-final}
\int_\Om W(\bG)\,\dd\bx \ge
\int_\Om W(\bF)\,\dd\bx
+
\sum_{j=1}^{k-1} \frac1{j!}
\int_\Om \DD^j W(\bF)\big[\bH,\bH,\ldots,\bH\big]\,\dd\bx
-c \int_\Om |\bH|^k\,\dd\bx,
\ee
\n where $\bH=\bG-\bF$, $\bF=\bF(\bx)$, $\bG=\bG(\bx)$, and, e.g.,
$W(\bF)=W(\bx,\bF(\bx))$.
\end{lemma}

\begin{proof} Fix $M>0$ and $\bF\in L^\infty(\Om;\MN)$.
Let $\bG\in\BMO(\Om;\MN)$ satisfy
$||\bG-\bF||_{\BMO} <M$.   We first note that
\eqref{eqn:BMO-in-Lq} in Proposition~\ref{thm:main-2} yields
\be\label{eqn:H-in-Lq}
\bH:=\bG-\bF\in L^q(\Om;\MN) \ \text{ for every $q\ge1$},
\ee
\n while (H3) together with the fact that $\bF$ is in $L^\infty$
yields (see Remark~\ref{rem:growth}), for some $C>0$
and $a.e.~\bx\in\Om$,
\be\label{eqn:DjW-bounded}
 \big|\DD^jW\big(\bx,\bF(\bx)\big)\big|\le C, \quad j=0,1,\ldots,k-1.
\ee
\n Consequently, \eqref{eqn:H-in-Lq} and \eqref{eqn:DjW-bounded}
yield,  for every $q\ge1$,
\be\label{eqn:DjW-in-Lq}
\bx\mapsto
\DD^j W\big(\bx,\bF(\bx)\big)
\big[\bH(\bx),\bH(\bx),\ldots,\bH(\bx)\big]
\in L^q(\Om;\MN),
\ee
\n for $j=0,1,\ldots,k-1$.


   Next, by Taylor's theorem for the function $\bA\mapsto W(\cdot,\bA)$,
for almost every $\bx\in\Om$,
\be\label{eqn:Tay-1}
\begin{aligned}
 W(\bG)&=  W(\bF) +
\sum_{j=1}^{k-1} \frac1{j!}\DD^j W(\bF)\big[\bH,\bH,\ldots,\bH\big]
+ R(\bF;\bH),\\
R(\bF;\bH)&:=\int_0^1 \frac{(1-t)^{k-1}}{(k-1)!}
\DD^kW(\bF+t\bH)\big[\bH,\bH,\ldots,\bH\big]\,\dd t.
\end{aligned}
\ee
\n We note that hypothesis (H3) together with the inequality
$|a+b|^r\le c_r (|a|^r+|b|^r)$, $c_r=\max\{1,2^{r-1}\}$,
and the fact that $t\in[0,1]$ gives us
\be\label{eqn:Dk-bounded}
|\DD^kW(\bF+t\bH)|\le c_k\big(1+|\bF+t\bH|^r\big)
\le
c_k+c_kc_r\big(|\bF|^r+|\bH|^r\big)
\ee
\n and hence the absolute value of the integrand in \eqref{eqn:Tay-1}$_2$
is bounded by $c_k/(k-1)!$ times
\be\label{eqn:Dk-bounded-2}
|\bH|^k\big(1+ c_r||\bF||^r_\infty\big)+c_r|\bH|^{k+r}.
\ee


We next integrate \eqref{eqn:Tay-1}$_{1}$ and \eqref{eqn:Tay-1}$_{2}$
over $\Om$  to get, in view
of \eqref{eqn:DjW-in-Lq}, \eqref{eqn:Dk-bounded}, and
\eqref{eqn:Dk-bounded-2},
\be\label{eqn:Tay-2}
\int_\Om W(\bG)\,\dd\bx
=
\int_\Om W(\bF)\,\dd\bx
+
\sum_{j=1}^{k-1} \frac1{j!}
\int_\Om \DD^j W(\bF)[\bH,\bH,\ldots,\bH]\,\dd\bx +
\int_\Om  R(\bF;\bH)\,\dd\bx
\ee
\n and
\be\label{eqn:Tay-3}
\begin{aligned}
\int_\Om  R(\bF;\bH)\,\dd\bx
&\le C_1\int_\Om |\bH|^k\,\dd\bx
+
C_2\int_\Om|\bH|^{k+r} \,\dd\bx\\
&\le
\big(C_1+C_2J_2^{k+r}||\bH||^r_{\BMO}\big)\int_\Om |\bH|^k\,\dd\bx,
\end{aligned}
\ee
\n where we have made use of \eqref{eqn:RH} of Proposition~\ref{thm:main-2}
with $p=k$ and $q=k+r$, $C_2:=c_kc_r/(k-1)!$, and
$C_1:=c_k(1+c_r||\bF||^r_\infty)/(k-1)!$.  The desired result,
\eqref{eqn:Taylor-final}, now follows from \eqref{eqn:Tay-2} and
\eqref{eqn:Tay-3}.
\end{proof}

\section{The Second Variation and
\texorpdfstring{$\BMO$}{BMO} Local Minimizers.}\label{sec:BMO-local}


We take
\[
\partial \Omega = \overline{\sD} \cup \overline{\sS}\ \
\text{with $\sD$ and $\sS$ relatively open and }
\sD\cap\sS=\varnothing.
\]
\n If $\sD\ne\varnothing$ we assume that a Lipschitz-continuous
function $\bd:\sD\to\R^N$
is prescribed. We define
\[  
W^{1,\BMO}(\Om;\R^N):=
\{\bu\in W^{1,1}(\Omega;\R^N): \grad\bu\in \BMO(\Omega;\MN)\}
\]  
\n and denote the set of \emph{Admissible Mappings} by
\[
\AM:=\{\bu\in W^{1,\BMO}(\Omega;\R^N):
\text{\ $\bu=\bd$ on $\sD$ or
$\langle\bu\rangle_\Omega=\mathbf{0}\, $ if $\, \sD=\varnothing$}\}.
\]
\n The \emph{energy} of $\bu\in\AM$ is defined by
\be\label{eqn:energy}
\E(\bu):=\int_\Omega W\big(\bx,\grad\bu(\bx)\big)\,\dd\bx,
\ee
\n where $W$ is given by Hypothesis~\ref{def:W-TT} with $k=3$.
We shall assume that we are given a
$\bu\e\in\AM$ that is a weak
solution of the \emph{Euler-Lagrange equations}
corresponding to \eqref{eqn:energy}, i.e.,
\be\label{eqn:ES}
0=\int_\Omega
\DD W\big(\bx,\grad\bu\e(\bx)\big)[\grad\bw(\bx)]
 \,\dd\bx,
\ee
\n for all \emph{variations} $\bw\in\Var$, where
\[
\Var:=\{\bw\in W^{1,2}(\Omega;\R^N):
\bw=\mathbf{0} \text{ on $\sD$  or
$\langle\bw\rangle_\Omega=\mathbf{0}$ if $\sD=\varnothing$}\}.
\]


\begin{theorem}\label{thm:SVP=LMBMO} Let $W$ satisfy
Hypothesis~\ref{def:W-TT} with $k=3$.  Suppose that
$\bu\e\in\AM\cap\, W^{1,\infty}(\Om;\R^N)$ is a weak solution of
\eqref{eqn:ES}
that satisfies, for some $a>0$, 
\be\label{eqn:small+SV}
\int_\Omega \DD^2 W(\grad\bu\e)
\big[\grad\bz, \grad\bz\big]\,\dd\bx
\ge
4a\int_\Omega |\grad\bz|^2\,\dd\bx\
\text{ for all $\bz\in\Var$}.
\ee
\n Then there exists a $\delta>0$ such that
any $\bv\in\AM$ that satisfies
\be\label{eqn:small+B}
||\grad\bv-\grad\bu\e||_{\BMO}<\delta
\ee
\n will also satisfy
\be\label{eqn:E-taylor-repeat-2}
\E(\bv)\ge \E(\bu\e)
+a\int_\Omega|\grad\bv-\grad\bu\e|^2\,\dd\bx.
\ee
\n  In particular, any $\bv\not\equiv\bu\e$ that satisfies
\eqref{eqn:small+B} will have strictly greater
energy than $\bu\e$.
\end{theorem}


\begin{remark}\label{rem:extensions}  1.~The theorem's conclusions remain
valid if one subtracts $\int_\Om \bb(\bx)\cdot\bu(\bx)\,\dd\bx$ and
$\int_{\sS}\bs(\bx)\cdot\bu(\bx)\,\dd{\mathcal{S}_\bx}$
from $\E$.  2.~Fix $q>2$.  Then inequality \eqref{eqn:RH} in
Proposition~\ref{thm:main-2} together  with
\eqref{eqn:E-taylor-repeat-2} yields a constant $\hat{j}=\hat{j}(q)$
such that any $\bv\in\AM$ that satisfies  \eqref{eqn:small+B} will
also satisfy
\[ 
\E(\bv)\ge \E(\bu\e)
+\hat{a}\hat{j}\delta^{2-q} \int_\Omega|\grad\bv-\grad\bu\e|^q\,\dd\bx.
\] 
\end{remark}

\begin{remark} The conclusions of Theorem~\ref{thm:SVP=LMBMO} remain
valid if we replace the assumption that $\bu\e$ is a weak solution of
\eqref{eqn:ES} by the assumption that $\bu\e$ is a weak relative minimizer
of $\E$, i.e., $\E(\bv)\ge \E(\bu\e)$ for all
$\bv\in\AM\cap\, W^{1,\infty}(\Om;\R^N)$
with $\|\grad\bv-\grad\bu\e\|_\infty$ sufficiently small.
\end{remark}


\begin{proof}[Proof of Theorem~\ref{thm:SVP=LMBMO}]
Let $\bu\e\in\AM$ be a weak solution of the Euler-Lagrange equations,
\eqref{eqn:ES}, that satisfies \eqref{eqn:small+SV}. Suppose that
$\bv\in\AM$ satisfies \eqref{eqn:small+B}
for some $\delta>0$ to be determined later and define
$\bw:=\bv-\bu\e\in\Var\cap W^{1,\BMO}$.  Then Lemma~\ref{lem:Taylor}
yields a constant $c>0$, such that
\be\label{eqn:final-in-1}
\E(\bv)\ge\E(\bue)+ 2\hat{k}\int_\Omega |\grad\bw|^2\,\dd\bx
-c\int_\Omega |\grad\bw|^3\,\dd\bx,
\ee
\n where we have made use of
\eqref{eqn:energy}--\eqref{eqn:small+SV}.      


We next note that inequality \eqref{eqn:RH} in
Proposition~\ref{thm:main-2}  (with $q=3$ and $p=2$) gives us
\be\label{eqn:final-in-2}
J^3||\grad\bw||_{\BMO}\int_\Omega |\grad\bw|^2\,\dd\bx
\ge \int_\Omega |\grad\bw|^3\,\dd\bx.
\ee
\n The desired inequality, \eqref{eqn:E-taylor-repeat-2}, now follows
from  \eqref{eqn:small+B}, \eqref{eqn:final-in-1}, and
\eqref{eqn:final-in-2} when $\delta$ is sufficiently small.
Finally,  $\E(\bv)>\E(\bue)$
is clear from \eqref{eqn:E-taylor-repeat-2} since $\Om$ is a connected
open region and either $\langle\bw\rangle_\Omega=\mathbf{0}$  or
$\bw=\mathbf{0}$ on $\sD\subset\partial\Om$.
\end{proof}

\begin{thebibliography}{99}

\bibitem{BM84} Ball, J. M.,  Marsden, J. E.:
Quasiconvexity at the boundary, positivity of the second variation
and elastic stability.
Arch. Ration. Mech. Anal. \textbf{86}, 251--277 (1984)

\bibitem{BN95} Brezis, H., Nirenberg, L.;
Degree theory and BMO. I. Compact manifolds without boundaries.
Selecta Math. (N.S.)  \textbf{1},  197--263 (1995)

\bibitem{BN96} Brezis, H., Nirenberg, L.;
Degree theory and BMO. II. Compact manifolds with boundaries.
With an appendix by the authors and Petru Mironescu.
Selecta Math. (N.S.) \textbf{2}, 309--368 (1996)


\bibitem{Ca17} Campos Cordero, J.:
Boundary regularity and sufficient conditions for strong
local minimizers.
J. Funct. Anal. \textbf{272}, 4513--4587 (2017)

\bibitem{DRS10} Diening, L.,
R$\overset{_\circ}{\mathrm{u}}$\v zi\v cka, M.,  Schumacher, K.:
A decomposition technique for John domains.
Ann. Acad. Sci. Fenn. Math.  \textbf{35}, 87--114  (2010)

\bibitem{EG92} Evans L. C., Gariepy, R. F.:
Measure Theory and Fine Properties of Functions.
CRC Press, Boca Raton (1992)

\bibitem{FS72} Fefferman, C., Stein, E. M.:
$H^p$ spaces of several variables.
Acta Math. \textbf{129},  137--193  (1972)

\bibitem{Fi92} Firoozye, N. B.:
Positive second variation and local minimizers in $\BMO$-Sobolev spaces.
Preprint no. 252, 1992, SFB 256, University of Bonn

\bibitem{Ge87}  Gehring, F. W.;
Uniform domains and the ubiquitous quasidisk.
Jahresber. Deutsch. Math.-Verein.  \textbf{89}, 88--103 (1987)

\bibitem{GO79} Gehring, F. W.,  Osgood, B. G.;
Uniform domains and the quasihyperbolic metric.
J. Analyse Math.  \textbf{36}, 50--74 (1979)

\bibitem{GM07} Grabovsky, Y.,  Mengesha, T.;
Direct approach to the problem of strong local minima in calculus of variations.
Calc. Var. Partial Differential Equations  \textbf{29}, 59--83  (2007)
[Erratum: \textbf{32}, 407--409  (2008)]

\bibitem{GM09} Grabovsky, Y.,  Mengesha, T.;
Sufficient conditions for strong local minimal: the case of $C^1$ extremals.
Trans. Amer. Math. Soc.  \textbf{361}, 1495--1541  (2009)

\bibitem{Gr09} Grafakos, L.:
Modern Fourier analysis. $3^{nd}$ edition. Springer, New York (2014)

\bibitem{HMT07} Hofmann, S.,  Mitrea, M., Taylor, M.:
Geometric and transformational properties of Lipschitz domains,
Semmes-Kenig-Toro domains, and other classes of finite
perimeter domains. J. Geom. Anal.  \textbf{17}, 593--647 (2007)

\bibitem{Iw82} Iwaniec, T.:
On $L^p$-integrability in PDEs and quasiregular mappings for large
exponents.
Ann. Acad. Sci. Fenn. Ser. A I Math. \textbf{7}, 301--322  (1982)

\bibitem{Jo72}  John, F.:
Uniqueness of non-linear elastic equilibrium for prescribed boundary
displacements and sufficiently small strains.
Commun. Pure Appl. Math.  \textbf{25}, 617--634  (1972)

\bibitem{JN61}	John, F.,  Nirenberg, L.:
On functions of bounded mean oscillation.
Commun. Pure Appl. Math.  \textbf{14}, 415--426 (1961)

\bibitem{Jo82} Jones, P. W.;
Extension theorems for BMO.
Indiana Univ. Math. J.  29,  41--66 (1980)


\bibitem{KT03} Kristensen, J.,  Taheri, A.:
Partial regularity of strong local minimizers in the
multi-dimensional calculus of variations.
Arch. Ration. Mech. Anal.  \textbf{170}, 63--89 (2003)

\bibitem{Mo66} Morrey, C. B., Jr.:
Multiple Integrals in the Calculus of Variations.
Springer, New York  (1966)

\bibitem{MS03}  M\"{u}ller, S., {\v{S}}ver{\'a}k, V.:
Convex integration for Lipschitz mappings and counterexamples
to regularity.
Ann. of Math. (2)  \textbf{157}, 715--742  (2003)

\bibitem{SS19} Spector, D. E., Spector, S. J.:
Uniqueness of equilibrium with sufficiently small strains.
Arch. Ration. Mech. Anal. \textbf{233}, 409--449 (2019)

\bibitem{Stein-1993} Stein, E. M.:
Harmonic Analysis: Real-variable Methods, Orthogonality, and
Oscillatory Integrals.
Princeton University Press, Princeton, NJ  (1993)

\end{thebibliography}
\end{document}